\documentclass[11pt]{article}

\usepackage[utf8]{inputenc}
\usepackage[T1]{fontenc}
\usepackage[margin=1in]{geometry}
\usepackage{mathptmx}

\usepackage{amsmath}
\usepackage{amssymb}
\usepackage{amsthm}
\usepackage{amsfonts}
\usepackage{mathtools}
\usepackage{mathrsfs}

\usepackage{booktabs}
\usepackage{array}
\usepackage{graphicx}

\usepackage{tikz}
\usetikzlibrary{arrows.meta,positioning,decorations.pathmorphing,decorations.pathreplacing,calc,fit,backgrounds,patterns}

\usepackage{xcolor}
\definecolor{orbitA}{HTML}{B85042}
\definecolor{orbitB}{HTML}{2F3C7E}
\definecolor{orbitC}{HTML}{5B8C5A}
\definecolor{baseworld}{HTML}{555555}
\definecolor{dashborder}{HTML}{B85042}

\usepackage{enumitem}

\usepackage{hyperref}
\usepackage{cleveref}
\usepackage{url}

\hypersetup{
    colorlinks=true,
    linkcolor=blue,
    citecolor=blue,
    urlcolor=blue,
    pdftitle={Modal Exchangeability: Centered Symmetry and the Credal Architecture of Kripke Frames},
    pdfauthor={Daniel Zantedeschi}
}

\crefname{theorem}{Theorem}{Theorems}
\crefname{lemma}{Lemma}{Lemmas}
\crefname{proposition}{Proposition}{Propositions}
\crefname{corollary}{Corollary}{Corollaries}
\crefname{definition}{Definition}{Definitions}
\crefname{example}{Example}{Examples}
\crefname{remark}{Remark}{Remarks}
\crefname{section}{Section}{Sections}
\crefname{figure}{Figure}{Figures}

\theoremstyle{plain}
\newtheorem{theorem}{Theorem}[section]
\newtheorem{lemma}[theorem]{Lemma}
\newtheorem{proposition}[theorem]{Proposition}
\newtheorem{corollary}[theorem]{Corollary}

\theoremstyle{definition}
\newtheorem{definition}[theorem]{Definition}
\newtheorem{example}[theorem]{Example}

\theoremstyle{remark}
\newtheorem*{remark}{Remark}

\DeclareMathOperator{\Aut}{Aut}
\DeclareMathOperator{\Stab}{Stab}
\DeclareMathOperator{\Sym}{Sym}

\newcommand{\Law}{\mathcal{L}}

\newcommand{\N}{\mathbb{N}}
\newcommand{\Z}{\mathbb{Z}}

\newcommand{\Acc}{\mathrm{Acc}}
\newcommand{\Orbit}{\mathcal{O}}
\newcommand{\FG}{\mathcal{F}^G}

\usepackage{setspace}
\usepackage[authoryear,round]{natbib}

\title{Modal Exchangeability:\\ Centered Symmetry and the Credal Architecture of Kripke Frames}
\author{Daniel Zantedeschi\thanks{Email: danielz@usf.edu. School of Information Systems, Muma College of Business, University of South Florida, Tampa, FL 33620, USA.}}
\date{\today}

\begin{document}

\maketitle

\begin{abstract}
De Finetti's representation theorem assumes a flat index set. We ask what happens when the index set carries modal structure, with possibilities organized into a Kripke frame. We define \emph{modal exchangeability} as invariance under accessibility-preserving automorphisms that fix a designated base world, and derive a representation theorem for countable frames. The orbit decomposition of the centered symmetry group governs the within-orbit structure: worlds in the same orbit are conditionally identically distributed, and on orbits satisfying a richness condition and countable infinitude they are conditionally i.i.d.\ given a rigid orbit-specific directing measure. Point-homogeneous $\mathsf{S5}$ frames yield a single de Finetti parameter; $\mathsf{S4}$ frames may admit multiple orbits, with the richer orbits carrying rigid directing measures and the remainder carrying only weaker invariant structure. Two applications follow. First, the orbit decomposition determines whether learning pools globally or remains orbit-local. Second, it supplies a mechanism for structural credal fine-graining indexed to orbit regions, distinct from hyperintensionality in the strict sense of distinguishing coextensive propositions.
\end{abstract}


\clearpage

\section{Introduction}
\label{sec:intro}

\subsection{The problem}

Rational agents form beliefs about structured possibility spaces. In many epistemically significant settings, propositions are not evaluated against a single flat domain but relative to possible worlds connected by an accessibility relation: knowledge, time, obligation, and metaphysical possibility all impose structure on the space over which uncertainty is defined. The question that motivates this paper is: what form must rational credence take when the space of epistemic possibilities carries modal structure?

De Finetti's theorem~\citep{deFinetti1937} provides the classical answer for \emph{unstructured} spaces: if an agent regards a sequence of observations as exchangeable (invariant under finite permutations), the joint distribution must be a mixture of i.i.d.\ processes. The theorem is powerful precisely because it derives rich probabilistic structure from a single symmetry assumption. But it presupposes that the underlying index set is \emph{flat}---all indices are interchangeable, and the space carries no structure beyond cardinality.

This is inadequate when the space of epistemic possibilities is modal. When propositions are evaluated relative to possible worlds connected by an accessibility relation, the relevant symmetries are not unrestricted permutations but structure-preserving transformations. This paper asks what happens to de Finetti's theorem when the index set is a Kripke frame~\citep{vanBenthem2010,vanBenthem2011}---when symmetry must respect accessibility structure and is centered on a designated base world. The question is therefore not only which worlds are possible, but what structure rational credence should take when possibility itself is organized by accessibility.

\subsection{Centered symmetry and modal exchangeability}

A Kripke frame $\langle W, R \rangle$ carries a natural symmetry group: its automorphisms, the bijections of $W$ that preserve $R$. But not all automorphisms are epistemically relevant. An agent situated at a base world $w_0$ occupies a privileged standpoint; the symmetries that matter are those that fix $w_0$ while permuting accessible worlds. These form the \emph{centered symmetry group} $G = \Stab_{w_0}$. The governing normative principle can be stated informally before the group theory arrives:

\medskip
\noindent\textbf{Centered Structural Indifference.} \emph{If two accessible worlds are related by an automorphism that preserves the frame and fixes the agent's standpoint, then---absent defeaters---a rational prior should not discriminate between them.}
\medskip

\noindent \emph{Modal exchangeability} makes this precise: the agent's probability measure over valuations is required to be $G$-invariant. Where de Finetti requires invariance under all permutations of an index set, modal exchangeability requires invariance under those permutations that respect accessibility.

\subsection{The main result and its interpretation}

A concrete example will anchor the discussion. Consider an $\mathsf{S4}$ frame with worlds $\{w_0\} \cup \{a_i : i \in \N\} \cup \{b_i : i \in \Z\}$: $w_0$ accesses every other world; the $a$-region is a complete reflexive graph; the $b$-region is a reflexive linear order indexed by $\Z$; no cross edges connect the two regions. The two regions are not isomorphic as relational structures, so no automorphism swaps them, and the centered symmetry group decomposes $\Acc(w_0)$ into two orbits, $O_a = \{a_i\}$ and $O_b = \{b_i\}$. The $a$-orbit satisfies a richness condition and supports a rigid directing measure in the classical de Finetti style; the $b$-orbit, whose automorphisms are the translations of $\Z$, supports only identical conditional marginals and a shift-stationary representation. One paper-wide contrast runs through the argument: the $a$-region gives de Finetti; the $b$-region gives transitive-but-not-exchangeable stationarity. This two-region frame (developed formally as Example~\ref{ex:S4}) will serve as a running illustration.

The general representation theorem (Theorem~\ref{thm:general}) shows that the orbit structure of $G$ acting on $\Acc(w_0)$ governs the within-orbit probabilistic decomposition: worlds in the same orbit are conditionally identically distributed; under a richness condition and countable infinitude, they are conditionally i.i.d.\ given an orbit-specific directing measure that is \emph{orbit-rigid} (shortened to ``rigid'' throughout, distinguished from Kripkean rigid designation in \S\ref{subsec:rigidity}). The main philosophical suggestion: modal structure constrains admissible credal architectures via symmetry.

Under point-homogeneous $\mathsf{S5}$ (where the centered symmetry group acts transitively on $\Acc(w_0)$), one recovers a single de Finetti parameter. Under $\mathsf{S4}$, the non-base accessible set may fragment---as in the running example---into multiple orbits with \emph{local} de Finetti representations and no symmetry-forced pooling across them. Weaker systems ($\mathsf{K}$, $\mathsf{T}$) do not determine a canonical architecture; the form is frame-relative. Two philosophical applications follow. First, the S4/S5 distinction controls whether learning pools globally or only orbit-locally. Second, the orbit decomposition provides a mechanism for structural credal fine-graining indexed to orbit regions, distinct from hyperintensionality in the strict sense of distinguishing coextensive propositions.

\subsection{Positioning}

The paper continues the line of work begun by \citet{Diaconis1977} and \citet{SuppesZanotti1981} on symmetry and probabilistic representation. Diaconis showed that exchangeability has nontrivial representational consequences; Suppes and Zanotti asked when probabilistic explanations are possible. The present contribution addresses a structural prior question: when does modal structure license a single global directing parameter, and when does it yield only orbit-local parameters? The proof machinery (disintegration, Hewitt--Savage) is classical; the novelty lies in identifying which symmetry group is epistemically relevant on a centered Kripke frame and tracing the logic-sensitive consequences.

The framework should not be confused with logical pluralism in the sense of \citet{Bueno2009}. It does not quantify over alternative logics or claim that different modal systems correspond to different ``true'' logics. Unlike approaches that analyze logical consequence by quantifying over cases, the present framework works with the symmetries of a fixed modal structure: the relevant constraint is invariance under accessibility-preserving transformations, not unrestricted variation over candidate cases. The S4/S5 distinction enters because it governs which orbit structures are possible, not because it adjudicates between competing logics. The result is a conditional representation theorem, not a metalogical thesis.

The contribution sits within probabilistic symmetry theory. Its specific aim is to identify the epistemically relevant symmetry group on a centered Kripke frame and to trace the representational consequences of invariance under that group. The orbit-rigidity results support structural credal fine-graining indexed to orbit regions, generated by the symmetry geometry of the frame rather than by enrichment of the semantic objects.

Modal structure does not merely constrain possibility; it constrains the architecture of rational uncertainty.

\section{Modal Frames, Automorphisms, and Modal Exchangeability}
\label{sec:framework}

This section introduces the formal apparatus: frames, automorphisms, stabilizers, orbits, and the modal exchangeability condition.

\subsection{Modal frames and accessible clusters}

\begin{definition}[Modal frame]
A \emph{modal frame} is a pair $\langle W, R \rangle$ where $W$ is a non-empty set of \emph{worlds} and $R \subseteq W \times W$ is an \emph{accessibility relation}.
\end{definition}

No conditions are imposed on $R$ at this stage. Specific modal logics correspond to specific constraints: $\mathsf{K}$ imposes none, $\mathsf{T}$ requires reflexivity, $\mathsf{S4}$ requires reflexivity and transitivity (a preorder), and $\mathsf{S5}$ requires reflexivity, symmetry, and transitivity (an equivalence relation); see \cref{fig:frames} for representative examples. Throughout, we assume $W$ is countable, ensuring that the space of valuations is a standard Borel space and enabling measure-theoretic arguments without strong choice principles.

\begin{definition}[Non-base accessible worlds]
For a designated base world $w_0 \in W$, the set of \emph{non-base accessible worlds} is $\Acc(w_0) := \{w \in W : w \neq w_0 \text{ and } R(w_0, w)\}$.
\end{definition}

The base world $w_0$ is excluded because the stabilizer $G = \Stab_{w_0}$ fixes $w_0$; including $w_0$ in the domain of the $G$-action would prevent transitivity claims from ever holding. In an $\mathsf{S5}$ frame, $\Acc(w_0)$ consists of the non-base members of $w_0$'s equivalence class. In an $\mathsf{S4}$ frame, it is the non-base portion of the upward-closed set above $w_0$. In a $\mathsf{K}$ frame, it may have no special structural properties.

\begin{figure}[ht]
\centering
\resizebox{\textwidth}{!}{%
\begin{tikzpicture}[
    world/.style={circle, draw=black!70, thick, fill=orbitA!70, text=white, minimum size=9mm, inner sep=0pt, font=\footnotesize\bfseries},
    worldB/.style={circle, draw=black!70, thick, fill=orbitB!70, text=white, minimum size=9mm, inner sep=0pt, font=\footnotesize\bfseries},
    worldMix/.style={circle, draw=black!70, thick, fill=orbitA!40, minimum size=9mm, inner sep=0pt, font=\footnotesize\bfseries},
    base/.style={circle, draw=baseworld, very thick, fill=baseworld, text=white, minimum size=9.5mm, inner sep=0pt, font=\footnotesize\bfseries},
    arr/.style={-{Stealth[length=5pt,width=4pt]}, thin, black!40},
    darr/.style={{Stealth[length=5pt,width=4pt]}-{Stealth[length=5pt,width=4pt]}, thin, black!40},
    framelabel/.style={font=\bfseries\small},
]

\begin{scope}[shift={(0,0)}]
  \node[framelabel] at (1.4,3.3) {$\mathsf{S5}$: Equivalence};
  \begin{scope}[on background layer]
    \draw[darr] (0,2.5) -- (2.8,2.5);
    \draw[darr] (0,0.3) -- (2.8,0.3);
    \draw[darr] (0,2.5) -- (0,0.3);
    \draw[darr] (2.8,2.5) -- (2.8,0.3);
    \draw[darr] (0,2.5) -- (2.8,0.3);
    \draw[darr] (2.8,2.5) -- (0,0.3);
    \draw[darr] (1.4,1.4) -- (0,2.5);
    \draw[darr] (1.4,1.4) -- (2.8,2.5);
    \draw[darr] (1.4,1.4) -- (0,0.3);
    \draw[darr] (1.4,1.4) -- (2.8,0.3);
  \end{scope}
  \node[base] (s5w0) at (1.4,1.4) {$w_0$};
  \node[world] (s5w1) at (0,2.5) {$w_1$};
  \node[world] (s5w2) at (2.8,2.5) {$w_2$};
  \node[world] (s5w3) at (0,0.3) {$w_3$};
  \node[world] (s5w4) at (2.8,0.3) {$w_4$};
  \node[font=\scriptsize, text=black!60, align=center] at (1.4,-0.6) {fully connected\\single cluster};
\end{scope}

\begin{scope}[shift={(4.8,0)}]
  \node[framelabel] at (1.8,3.3) {$\mathsf{S4}$: Preorder};
  \begin{scope}[on background layer]
    \draw[arr] (0,1.4) -- (1.2,2.5);
    \draw[arr] (0,1.4) -- (3.0,2.5);
    \draw[arr] (0,1.4) -- (1.2,0.3);
    \draw[arr] (0,1.4) -- (3.0,0.3);
    \draw[darr] (1.2,2.5) -- (3.0,2.5);
    \draw[darr] (1.2,0.3) -- (3.0,0.3);
    \draw[arr] (1.2,2.5) -- (1.2,0.3);
    \draw[arr] (3.0,2.5) -- (3.0,0.3);
  \end{scope}
  \node[base] (s4w0) at (0,1.4) {$w_0$};
  \node[world] (s4w1) at (1.2,2.5) {$w_1$};
  \node[world] (s4w2) at (3.0,2.5) {$w_2$};
  \node[worldB] (s4w3) at (1.2,0.3) {$w_3$};
  \node[worldB] (s4w4) at (3.0,0.3) {$w_4$};
  \node[font=\scriptsize, text=black!60, align=center] at (1.5,-0.6) {directed hierarchy\\distinct clusters};
\end{scope}

\begin{scope}[shift={(9.6,0)}]
  \node[framelabel] at (1.4,3.3) {$\mathsf{T}$: Reflexive};
  \begin{scope}[on background layer]
    \draw[arr] (0,1.4) -- (1.4,2.5);
    \draw[arr] (1.4,2.5) -- (2.8,1.4);
    \draw[arr] (2.8,1.4) -- (0,1.4);
    \draw[arr] (0,1.4) -- (1.4,0.3);
  \end{scope}
  \node[base] (tw0) at (0,1.4) {$w_0$};
  \node[world] (tw1) at (1.4,2.5) {$w_1$};
  \node[worldB] (tw2) at (2.8,1.4) {$w_2$};
  \node[worldMix] (tw3) at (1.4,0.3) {$w_3$};
  \node[font=\scriptsize, text=black!60, align=center] at (1.4,-0.6) {reflexive only\\no transitivity};
\end{scope}

\begin{scope}[shift={(13.6,0)}]
  \node[framelabel] at (1.1,3.3) {$\mathsf{K}$: Arbitrary};
  \begin{scope}[on background layer]
    \draw[arr] (0,1.8) -- (1.1,2.5);
    \draw[arr] (1.1,2.5) -- (2.2,1.4);
    \draw[arr] (1.1,0.3) -- (0,1.8);
  \end{scope}
  \node[base] (kw0) at (0,1.8) {$w_0$};
  \node[world] (kw1) at (1.1,2.5) {$w_1$};
  \node[worldB] (kw2) at (2.2,1.4) {$w_2$};
  \node[worldMix] (kw3) at (1.1,0.3) {$w_3$};
  \node[font=\scriptsize, text=black!60, align=center] at (1.1,-0.6) {no constraints\\chaotic structure};
\end{scope}

\end{tikzpicture}%
}
\caption{Four types of Kripke frames illustrating the modal hierarchy. The base world $w_0$ is shown in dark gray. In $\mathsf{S5}$ (equivalence), all worlds are mutually accessible and uniformly colored---a single cluster. In $\mathsf{S4}$ (preorder), directed accessibility creates distinct regions (terracotta and navy). In $\mathsf{T}$ (reflexive), structure begins to dissolve: accessibility is irregular and worlds carry heterogeneous structural roles. In $\mathsf{K}$ (arbitrary), no constraints govern accessibility. As one descends the hierarchy, the frame's internal organization weakens and the automorphism group becomes less constrained.}
\label{fig:frames}
\end{figure}

\subsection{Automorphisms and stabilizers}

\begin{definition}[Frame automorphism]
An \emph{automorphism} of $\langle W, R \rangle$ is a bijection $\pi: W \to W$ such that $R(w, v) \iff R(\pi(w), \pi(v))$ for all $w, v \in W$. We write $\Aut\langle W, R \rangle$ for the automorphism group.
\end{definition}

Automorphisms are the symmetries of the frame. They permute worlds while preserving the accessibility structure, and they are the modal analogue of the permutations that appear in classical exchangeability. For $\mathsf{S5}$ frames, every automorphism permutes within equivalence classes; for $\mathsf{S4}$ frames, automorphisms preserve the preorder; for $\mathsf{K}$ frames, they preserve whatever relational structure happens to be present.

\begin{definition}[Centered symmetry group (stabilizer)]
The \emph{centered symmetry group} of $w_0 \in W$ is $G := \Stab_{w_0} = \{\pi \in \Aut\langle W, R \rangle : \pi(w_0) = w_0\}$.
\end{definition}

The centered symmetry group consists of automorphisms that fix the base world while permuting others. It encodes the symmetries available to an agent situated at $w_0$: the transformations of the frame that leave the agent's standpoint unchanged. Since $G$ preserves $R$ and fixes $w_0$, it permutes $\Acc(w_0)$: if $\pi \in G$ and $w \in \Acc(w_0)$, then $w \neq w_0$ and $R(w_0, w)$, so $\pi(w) \neq w_0$ and $R(w_0, \pi(w))$, giving $\pi(w) \in \Acc(w_0)$.

\begin{definition}[Orbits]
The \emph{$G$-orbits} in $\Acc(w_0)$ are the equivalence classes under the relation $v \sim v' \iff \exists \pi \in G: \pi(v) = v'$. We write $\Orbit = \{O_1, O_2, \ldots\}$ for the set of orbits.
\end{definition}

The orbit structure captures which non-base accessible worlds are ``structurally indistinguishable'' from $w_0$'s perspective (\cref{fig:stabilizer-orbits}). In a point-homogeneous $\mathsf{S5}$ frame where $G$ acts transitively on $\Acc(w_0)$, there is a single non-base orbit; in an $\mathsf{S4}$ frame, there may be several; in a $\mathsf{K}$ frame, the number of orbits can range from one to $|\Acc(w_0)|$.

\begin{figure}[ht]
\centering
\begin{tikzpicture}[
    world/.style={circle, draw=black!70, thick, fill=orbitA!70, text=white, minimum size=9mm, inner sep=0pt, font=\footnotesize\bfseries},
    worldB/.style={circle, draw=black!70, thick, fill=orbitB!70, text=white, minimum size=9mm, inner sep=0pt, font=\footnotesize\bfseries},
    worldC/.style={circle, draw=black!70, thick, fill=orbitC!70, text=white, minimum size=9mm, inner sep=0pt, font=\footnotesize\bfseries},
    base/.style={circle, draw=baseworld, very thick, fill=baseworld, text=white, minimum size=10mm, inner sep=0pt, font=\footnotesize\bfseries},
    arr/.style={-{Stealth[length=5pt,width=4pt]}, thin, black!35},
    autarr/.style={{Stealth[length=5pt,width=4pt]}-{Stealth[length=5pt,width=4pt]}, semithick, densely dashed, color=black!55},
    orbitbox/.style={draw=dashborder, dashed, thick, rounded corners=5pt, inner sep=10pt},
    orbitboxB/.style={draw=orbitB, dashed, thick, rounded corners=5pt, inner sep=10pt},
    orbitboxC/.style={draw=orbitC, dashed, thick, rounded corners=5pt, inner sep=10pt},
    wavey/.style={decorate, decoration={snake, amplitude=1.5pt, segment length=6pt}, gray!50, semithick},
]

\begin{scope}[on background layer]
  \draw[arr] (0,1.5) -- (2.8,3.6);
  \draw[arr] (0,1.5) -- (4.3,3.6);
  \draw[arr] (0,1.5) -- (5.8,3.6);
  \draw[arr] (0,1.5) -- (3.2,1.5);
  \draw[arr] (0,1.5) -- (5.2,1.5);
  \draw[arr] (0,1.5) -- (4.2,-0.2);
\end{scope}

\node[base] (w0) at (0,1.5) {$w_0$};

\draw[semithick, black!45, -{Stealth[length=5pt]}] (w0) to[out=160,in=220,looseness=5] (w0);
\node[font=\scriptsize\itshape, text=black!50, anchor=east] at (-0.75,1.5) {fixed by $G$};

\node[world] (a1) at (2.8,3.6) {$w_1$};
\node[world] (a2) at (4.3,3.6) {$w_2$};
\node[world] (a3) at (5.8,3.6) {$w_3$};

\node[worldB] (b1) at (3.2,1.5) {$w_4$};
\node[worldB] (b2) at (5.2,1.5) {$w_5$};

\node[worldC] (c1) at (4.2,-0.2) {$w_6$};

\begin{scope}[on background layer]
  \node[orbitbox, fill=orbitA!6, fit=(a1)(a2)(a3)] (boxA) {};
  \node[orbitboxB, fill=orbitB!6, fit=(b1)(b2)] (boxB) {};
  \node[orbitboxC, fill=orbitC!6, fit=(c1)] (boxC) {};
\end{scope}

\node[font=\small, text=orbitA!80!black, anchor=south] at (boxA.north) {Orbit $O_1$: $\Lambda_{O_1}$ rigid};
\node[font=\small, text=orbitB!80!black, anchor=west] at (boxB.east) {Orbit $O_2$: $\Lambda_{O_2}$ rigid};
\node[font=\small, text=orbitC!80!black, anchor=west] at (boxC.east) {Orbit $O_3$: singleton};

\draw[autarr, bend left=40] (a1) to node[above, font=\scriptsize, fill=white, inner sep=1pt] {$\pi$} (a2);
\draw[autarr, bend left=40] (a2) to node[above, font=\scriptsize, fill=white, inner sep=1pt] {$\pi$} (a3);

\draw[autarr, bend right=40] (b1) to node[below, font=\scriptsize, fill=white, inner sep=1pt] {$\pi'$} (b2);

\draw[wavey] ($(boxA.south)+(0.5,-0.05)$) -- ($(boxB.north)+(0.0,0.05)$);
\draw[wavey] ($(boxB.south)+(-0.3,-0.05)$) -- ($(boxC.north)+(0.0,0.05)$);

\node[font=\small\itshape, text=black!50, anchor=south west] at (-0.5,4.3) {$G = \Stab_{w_0}$};

\end{tikzpicture}
\caption{Stabilizer action and orbit decomposition. The base world $w_0$ (dark gray) is fixed by every $\pi \in G = \Stab_{w_0}$. The non-base accessible set $\Acc(w_0) = \{w_1, \ldots, w_6\}$ decomposes into three $G$-orbits, distinguished by color. Dashed arrows show stabilizer automorphisms permuting worlds within their orbit. Under (Ext) and countable infinitude, each such orbit carries a \emph{rigid} directing measure $\Lambda_O$, constant across its worlds. Wavy lines indicate that $G$-invariance does not force independence across orbits; the joint distribution of $\{\Lambda_O\}$ may be constrained by additional structure in the group action but is not determined by $G$-invariance alone.}
\label{fig:stabilizer-orbits}
\end{figure}

\subsection{Valuations and the automorphism action}

\begin{definition}[Valuation space]
Let $L = \{\ell_1, \ldots, \ell_k\}$ be a finite set of propositional atoms. A \emph{valuation} on $W$ is a function $V: W \to \{0,1\}^L$. The \emph{valuation space} is $\Omega = (\{0,1\}^L)^W$, equipped with the product $\sigma$-field $\mathcal{F}$.
\end{definition}

Since $W$ is countable and $L$ is finite, $(\Omega, \mathcal{F})$ is a standard Borel space.

\begin{definition}[Automorphism action on valuations]
For $\pi \in \Aut\langle W, R \rangle$ and $V \in \Omega$, define $(\pi \cdot V)(w) := V(\pi^{-1}(w))$.
\end{definition}

This action is measurable and preserves the product $\sigma$-field.

\subsection{Valuations and formula semantics}

The framework places a probability measure $P$ on the space of valuation profiles $V: W \to \{0,1\}^L$. Any propositional or modal formula $\varphi$ evaluated at a world $w$ in a valuation $V$ determines a measurable event: $\{V \in \Omega : M, w \models_V \varphi\}$, where $M$ is the Kripke model with frame $\langle W, R \rangle$ and valuation $V$. By standard Kripke semantics, propositional atoms are assigned truth values according to $V$, and modal formulas are evaluated via the accessibility relation $R$. The probability measure $P$ induces a probability on formulas via the event assignment: for any formula $\varphi$ and world $w$, the probability that $\varphi$ holds at $w$ under $P$ is $P(\{V : M, w \models_V \varphi\})$. In this framework, we study symmetry constraints on the underlying valuation measure $P$, not an alternative truth theory. The representation theorems characterize how modal exchangeability constrains $P$, which in turn constrains the induced probabilities on formulas.

\subsection{Modal exchangeability}

\begin{definition}[Modal exchangeability]
\label{def:world-exch}
A probability measure $P$ on $(\Omega, \mathcal{F})$ is \emph{modally exchangeable at $w_0$} if $P$ is $G$-invariant: $P(\pi^{-1}(A)) = P(A)$ for all $\pi \in G$ and $A \in \mathcal{F}$.
\end{definition}

Modal exchangeability expresses rational indifference with respect to modal structure. If two configurations of truth-values across worlds are related by an accessibility-preserving automorphism that fixes the base world, they must receive the same probability. The agent has no probabilistic basis for distinguishing worlds that occupy the same structural position in the frame.

When $R$ is an equivalence relation and $G$ contains all finite permutations of $\Acc(w_0)$, modal exchangeability reduces to classical exchangeability. In this sense, de Finetti's symmetry is the special case of modal exchangeability on $\mathsf{S5}$ frames.

\subsection{The extension property}

The general representation theorem yields conditional identical distributions within orbits. To obtain the stronger conclusion of conditional independence (the full de Finetti structure), the stabilizer must contain sufficiently many permutations.

\begin{definition}[Extension property]
\label{def:ext}
The stabilizer $G$ has the \emph{finite-support extension property} (Ext) on an orbit $O$ if: for any finite, pairwise-disjoint subsets $F_1, \ldots, F_m \subseteq O$ and $F'_1, \ldots, F'_m \subseteq O$ with $|F_i| = |F'_i|$ for all $i$, there exists $\pi \in G$ such that $\pi(F_i) = F'_i$ for all $i$.
\end{definition}

Informally, (Ext) says the orbit has ``enough symmetries'' for exchangeability to bite: any partial rearrangement of finitely many worlds extends to a full automorphism. The simplest sufficient condition is that the restricted action $G|_O$ be the full symmetric group---which holds whenever the induced subframe on $O$ carries only reflexive accessibility and all worlds in $O$ have identical incidence patterns to worlds outside $O$. More generally, (Ext) holds whenever the relational structure on $O$ is ultrahomogeneous in the Fra\"{i}ss\'{e} sense.\footnote{By Fra\"{i}ss\'{e}'s theorem~\citep{Fraisse1954,Hodges1993,TentZiegler2012}, a countable relational structure is ultrahomogeneous if and only if it is the Fra\"{i}ss\'{e} limit of an amalgamation class. For modal frames, an orbit satisfies (Ext) whenever the relational structure induced on it by $R$ is such a limit. See the proof of Proposition~\ref{prop:ext-sufficient} in \S\ref{sec:proofs} for details.} Both the $\mathsf{S5}$ frames of Theorem~\ref{thm:S5} and the two-orbit $\mathsf{S4}$ frame of our running example satisfy (Ext).

When (Ext) fails, the representation still yields identical conditional marginals within orbits, but not conditional i.i.d.\ A concrete failure:

\begin{example}[Linear-order orbit where (Ext) fails]
\label{ex:ext-failure-preview}
Consider an orbit with the structure of $(\Z, \leq)$: a set of worlds $\{a_i : i \in \Z\}$ on which $R$ is the standard linear order and the stabilizer consists of the translations $\pi_k(a_i) = a_{i+k}$. The action is transitive, but translations do not realize arbitrary finite permutations (swapping $a_0$ and $a_1$ reverses order), so (Ext) fails. The conditional law of $\{V(a_i)\}_{i \in \Z}$ given $\Xi$ is shift-stationary but not exchangeable. This is exactly the $b$-region of the running example (Example~\ref{ex:S4}).
\end{example}

\noindent Property (Ext) thus marks the boundary between \emph{statistical homogeneity} (within-orbit conditional i.i.d.\ given a directing law) and \emph{marginal indistinguishability} (identical conditional marginals from $G$-invariance alone). The three layers of within-orbit structure are summarized as follows:

\begin{proposition}[Three layers of within-orbit structure]
\label{prop:ext-role}
Within an orbit $O \subseteq \Acc(w_0)$, the representation of a modally exchangeable measure admits three layers, each corresponding to a strengthening of the input:
\begin{enumerate}[leftmargin=2em]
\item \textup{($G$-invariance alone)} Identical conditional marginals: $\Law(V(v) \mid \Xi) = \Law(V(v') \mid \Xi)$ for $v, v' \in O$.
\item \textup{($G$-invariance + (Ext))} Conditional exchangeability of $\{V(w)\}_{w \in O}$ given $\Xi$.
\item \textup{($G$-invariance + (Ext) + $|O| = \aleph_0$)} Conditional i.i.d.: $\{V(w)\}_{w \in O}$ are i.i.d.\ given a rigid directing measure $\Lambda_O$.
\end{enumerate}
\end{proposition}

\noindent The three layers are distinct; each is strictly stronger than the preceding one, and each corresponds to a specific theorem or lemma below.

\subsection{Orbit structure and conditional independence}

The orbit decomposition should not be conflated with conditional independence. By Proposition~\ref{prop:ext-role}, worlds in the same orbit have identical conditional marginals (layer 1), and under (Ext) they are conditionally exchangeable (layer 2); under (Ext) and countable infinitude, the stronger conditional i.i.d.\ representation obtains (layer 3). But within-orbit structure is all $G$-invariance delivers: it does not force conditional independence across orbits. For orbits in $\Orbit^\infty$, each carries its own rigid directing measure, but $G$-invariance alone does not determine the joint distribution of $\{\Lambda_O\}_{O \in \Orbit^\infty}$; any constraints on cross-orbit coupling must come from the agent's prior, from the latent invariant structure $\Xi$, or from richer structure in the group action. Orbit membership thus yields marginal equivalence and, under the additional conditions, exchangeability or i.i.d.\ structure within orbits---but not cross-orbit independence.

\section{General Representation over Arbitrary Frames}
\label{sec:general}

The theorem applies to any countable frame with any accessibility relation. The centered symmetry group partitions $\Acc(w_0)$ into orbits; within each orbit, worlds are probabilistically interchangeable. On orbits satisfying (Ext) and countable infinitude, this upgrades to conditional i.i.d.\ given a rigid orbit-specific directing measure; other orbits support only weaker invariant structure. $G$-invariance does not force independence across orbits.

\begin{theorem}[General modal de Finetti]
\label{thm:general}
Let $\langle W, R \rangle$ be a countable modal frame, $w_0 \in W$ a base world, $G = \Stab_{w_0}$ the stabilizer, and $P$ a modally exchangeable probability measure on $(\Omega, \mathcal{F})$. Let $\Orbit = \{O_1, O_2, \ldots\}$ denote the $G$-orbits in $\Acc(w_0)$, and let
\[
\Orbit^\infty := \{O \in \Orbit : |O| = \aleph_0 \text{ and } G|_O \text{ has property (Ext)}\}.
\]
Then:
\begin{enumerate}[leftmargin=2em]
\item \textbf{Ergodic decomposition.} There exists a $\FG$-measurable random variable $\Xi$ such that $P = \int P(\cdot \mid \Xi = \xi) \, P_\Xi(d\xi)$, with each fiber $P(\cdot \mid \Xi = \xi)$ being $G$-ergodic. \emph{(Informally: a latent invariant random parameter indexing ergodic components.)}

\item \textbf{Identical distributions within orbits.} For any $v, v'$ in the same $G$-orbit (regardless of (Ext) or cardinality), $\Law(V(v) \mid \Xi) = \Law(V(v') \mid \Xi)$ holds $P$-a.s.

\item \textbf{Conditional i.i.d.\ on $\Orbit^\infty$.} For each $O \in \Orbit^\infty$, there exists a random probability measure $\Lambda_O$ on $\{0,1\}^L$ such that, conditional on $\Xi$, the valuations $\{V(w)\}_{w \in O}$ are i.i.d.\ with law $\Lambda_O$. (For finite orbits, exchangeability yields only a finite-exchangeability approximation; for infinite orbits on which (Ext) fails, Hewitt--Savage does not apply and no directing measure $\Lambda_O$ is obtained from this theorem.)

\item \textbf{Cross-orbit independence not forced.} $G$-invariance does not force independence across orbits in $\Orbit^\infty$; additional structure in the group action may still constrain the joint distribution of $\{\Lambda_O\}_{O \in \Orbit^\infty}$. Coupling between $\Orbit^\infty$ and the remainder $\Orbit \setminus \Orbit^\infty$ is encoded in the latent invariant $\Xi$ rather than in orbit-wise directing measures.

\item \textbf{Rigidity on $\Orbit^\infty$.} For each $O \in \Orbit^\infty$, the directing measure $\Lambda_O$ is $\FG$-measurable: for any $\pi \in G$ with $\pi(O) = O$, $\Lambda_O(\pi \cdot V) = \Lambda_O(V)$ holds $P$-a.s. The measure does not depend on the choice of world within $O$.
\end{enumerate}
\end{theorem}

\begin{proof}[Proof sketch]
Disintegrate $P$ over the invariant $\sigma$-field $\FG$ (standard Borel, so this requires only DC~\citep[Theorem 452I]{Fremlin2006}), yielding an ergodic decomposition $P = \int P(\cdot \mid \Xi = \xi)\,P_\Xi(d\xi)$. Part~2: any $v, v'$ in the same orbit are interchangeable by some $\pi \in G$, and $G$-invariance of $P$ and $\Xi$ gives identical conditional laws. Part~3: for $O \in \Orbit^\infty$, (Ext) ensures that $G$ induces all finite permutations within $O$, so the conditional distribution is classically exchangeable; Hewitt--Savage~\citep{Hewitt1955} (applicable because $O$ is countably infinite) then delivers $\Lambda_O$. Part~4: no element of $G$ maps one orbit to another, so $G$-invariance imposes no independence requirement across orbits. Part~5: for $O \in \Orbit^\infty$, $\Lambda_O$ is $\FG$-measurable by construction, hence rigid. A step-by-step version appears in \S\ref{sec:proofs}.
\end{proof}

\begin{remark}[Orbit structure as the primary within-orbit invariant]
The orbit structure is the primary invariant governing within-orbit probabilistic structure. Two frames validating the same modal logic may have different orbit structures and hence different representation forms. The modal logic constrains which orbit structures are possible (e.g., point-homogeneous $\mathsf{S5}$ yields a single orbit); the orbit geometry then determines the within-orbit decomposition, while cross-orbit relationships depend on additional structure.
\end{remark}

\begin{remark}[Interpretive consequences]
\label{rmk:interpretive-preview}
Given modal exchangeability, the representation theorem implies an orbit-wise indifference constraint: within any orbit satisfying (Ext) and countable infinitude, the directing measure $\Lambda_O$ treats all worlds identically. Observations at worlds in orbit $O$ update the conditional law there but are uninformative about $O' \neq O$ unless the agent's prior couples them. In the running example, the $a$-orbit admits a rigid directing measure while the $b$-orbit admits only a shift-stationary representation; across regions, no symmetry constraint forces pooling. These consequences are developed in \S\ref{sec:philosophy}.
\end{remark}

\begin{remark}[Relation to Crane--Towsner]
The orbit decomposition corresponds to the ``definable partition'' in \citet{CraneTowsner2018}'s theory of relatively exchangeable structures. The result of this section can be viewed as a modal specialization of that framework, with the novelty lying in the centered-stabilizer construction and its epistemic interpretation. See Appendix~\ref{app:representation} for a comparison.
\end{remark}

\begin{corollary}[Bernoulli parameterization]
\label{cor:bernoulli}
Under the conditions of Theorem~\ref{thm:general}, if the $L$ coordinates are conditionally independent within each world given $\Lambda_O$, then $\Lambda_O = \bigotimes_{\ell=1}^k \mathrm{Bernoulli}(\Theta_{O,\ell})$ for rigid parameters $\Theta_O = (\Theta_{O,1}, \ldots, \Theta_{O,k}) \in [0,1]^L$.
\end{corollary}

\paragraph{Philosophical significance.}
Theorem~\ref{thm:general} traces a pipeline from modal structure to probabilistic structure. The accessibility relation yields the automorphism group; the centered stabilizer yields the orbit geometry; the orbit geometry constrains what invariance can require of credence. Orbits in $\Orbit^\infty$ carry rigid directing measures and support an orbit-wise de Finetti representation; other orbits carry only weaker invariant structure (identical conditional marginals, and, where the group action is transitive but order-preserving, a shift-stationary conditional law).

The modal structure shapes credal architecture: it constrains where orbit-wise parameters can arise, which worlds share a parameter, and where Bayesian learning pools. The classical de Finetti theorem is the limiting case in which the modal structure is trivial, the group is maximal, and the architecture is globally homogeneous.

\section{The Modal Hierarchy}
\label{sec:hierarchy}

The general theorem applies to all countable frames. What distinguishes the modal logics is the type of orbit structure they permit.

\subsection{Point-homogeneous S5: global de Finetti}

\begin{definition}[Point-homogeneity]
\label{def:point-hom}
The frame is \emph{point-homogeneous at $w_0$} when $G$ acts transitively on $\Acc(w_0)$: for any $v, v' \in \Acc(w_0)$, some $\pi \in G$ has $\pi(v) = v'$.
\end{definition}

\begin{theorem}[Modal de Finetti: S5]
\label{thm:S5}
Let $\langle W, R \rangle$ be an $\mathsf{S5}$ frame with $R$ an equivalence relation and $P$ a modally exchangeable probability measure on $\Omega$. Assume the frame is point-homogeneous at $w_0$, the stabilizer $G$ has property \textup{(Ext)} on $\Acc(w_0)$, and $\Acc(w_0)$ is countably infinite.

Let $\widetilde{V} := V|_{\Acc(w_0)}$ denote the valuation restricted to $\Acc(w_0)$, and let $P_{\mathrm{acc}}$ be the law of $\widetilde{V}$ under $P$ (equivalently, the marginal of $P$ on the $\sigma$-field generated by $\{V(w) : w \in \Acc(w_0)\}$). Then there exists a probability measure $\mu$ on $\mathcal{M}(\{0,1\}^L)$ such that for every measurable $A$,
\[
P_{\mathrm{acc}}(\widetilde{V} \in A) = \int_{\mathcal{M}} P_\Lambda(\widetilde{V} \in A) \, \mu(d\Lambda),
\]
where $P_\Lambda$ is the product measure under which $\{\widetilde{V}(w)\}_{w \in \Acc(w_0)}$ are i.i.d.\ with common law $\Lambda$. \textup{(}For finite $\Acc(w_0)$, the same argument yields only a finite-exchangeability approximation; the exact i.i.d.\ mixture requires countable infinitude.\textup{)}
\end{theorem}

\begin{proof}
Under point-homogeneity, $\Acc(w_0)$ is a single $G$-orbit. Countable infinitude plus (Ext) lets us apply Theorem~\ref{thm:general}(3): the conditional distribution of $\{V(w)\}_{w \in \Acc(w_0)}$ given $\Xi$ is exchangeable, and Hewitt--Savage yields a directing measure $\Lambda(\Xi)$. Defining $\mu$ as the distribution of $\Lambda(\Xi)$ gives the representation for $P_{\mathrm{acc}}$.
\end{proof}

The $\mathsf{S5}$ case is not merely the setting in which de Finetti's theorem reappears with a new label. It is the case of \emph{global epistemic homogeneity}: a single directing measure means that all non-base accessible worlds are governed by one shared parameter, and observations at any such world update the same parameter. Learning pools completely across $\Acc(w_0)$, and the agent's posterior converges---under standard Bayesian consistency conditions---to the true directing law from data gathered anywhere in the accessible cluster.\footnote{Since the proofs use only disintegration on standard Borel spaces~\citep[Theorem 452I]{Fremlin2006} and the Hewitt--Savage theorem~\citep{Hewitt1955}, the entire argument goes through in ZF + DC.}

\subsection{S4: local de Finetti with rigid directing measures}

\begin{theorem}[Modal de Finetti: S4 / Orbit-wise decomposition]
\label{thm:S4}
Let $\langle W, R \rangle$ be an $\mathsf{S4}$ frame and $P$ a modally exchangeable probability measure on $\Omega$. Let $\Orbit = \{O_1, O_2, \ldots\}$ denote the $G$-orbits in $\Acc(w_0)$. Let $\Orbit^\infty \subseteq \Orbit$ be the set of countably infinite orbits for which the restricted action $G|_O$ has property \textup{(Ext)}.

Then there exist random probability measures $\{\Lambda_O : O \in \Orbit^\infty\}$ on $\{0,1\}^L$ such that:
\begin{enumerate}[leftmargin=2em]
\item For each $O \in \Orbit^\infty$, conditional on $\Lambda_O$, the valuations $\{V(w)\}_{w \in O}$ are i.i.d.\ with law $\Lambda_O$ (exactly for infinite orbits; the finite-orbit case yields a finite-exchangeability approximation).
\item For orbits outside $\Orbit^\infty$ (finite orbits, or orbits on which (Ext) fails), $G$-invariance yields identical conditional marginals but not necessarily conditional i.i.d.
\item $G$-invariance does not force independence across orbits; however, additional structure in the group action may impose constraints on joint distributions across orbits.
\item For each $O \in \Orbit^\infty$, the directing measure $\Lambda_O$ is rigid: $\FG$-measurable, constant across $O$.
\end{enumerate}
\end{theorem}

The $\mathsf{S4}$ case changes the geometry of learning. Orbits in $\Orbit^\infty$ each carry a rigid directing measure, and observations there update only that orbit's parameter---by symmetry alone---without informing any other orbit. Orbits outside $\Orbit^\infty$ receive only identical conditional marginals and, where transitive but order-preserving group actions are involved, a stationary conditional law. Cross-orbit learning can occur through hierarchical prior coupling or through richer group-action structure, but such coupling is a modeling choice, not a direct consequence of modal exchangeability.

\begin{proof}
Immediate from Theorem~\ref{thm:general}: parts (2) applies to every orbit, while parts (3)--(5) apply to orbits in $\Orbit^\infty$.
\end{proof}

Under $\mathsf{S4}$, directed accessibility may fragment $\Acc(w_0)$ into multiple orbits (depending on frame structure). Those in $\Orbit^\infty$ support their own rigid directing measures; others support weaker invariant structure. $G$-invariance does not relate the directing measures of different orbits in $\Orbit^\infty$. Observations within such an orbit update a single shared parameter; observations across orbits do not pool unless the agent's prior or additional group-action structure introduces coupling.

\subsection{Weak modal logics and the limits of logic-sensitive credence}

$\mathsf{T}$ and $\mathsf{K}$ mark the boundary where modal axioms become too weak to force a canonical orbit structure.

\begin{proposition}[T frames]
\label{prop:T}
Let $\langle W, R \rangle$ be a countable $\mathsf{T}$-frame and $P$ a modally exchangeable measure. The general representation \textup{(Theorem~\ref{thm:general})} applies. However, the class of $\mathsf{T}$-frames does not determine a canonical orbit structure: the orbit decomposition of $\Acc(w_0)$ is frame-dependent.
\end{proposition}

Reflexivity has epistemic significance---$\Box \varphi \to \varphi$ ensures that what is known is true---but it imposes no systematic constraint on the orbit decomposition of $\Acc(w_0)$. Individual $\mathsf{T}$-frames \emph{can} support interesting credal structure, but such structure is a property of the specific frame, not a consequence of the $\mathsf{T}$ axiom.

\begin{proposition}[K frames]
\label{prop:K}
Let $\langle W, R \rangle$ be a countable $\mathsf{K}$-frame and $P$ a modally exchangeable measure. The general representation \textup{(Theorem~\ref{thm:general})} applies. However, the class of $\mathsf{K}$-frames does not determine a canonical orbit structure: any number of orbits, from one to $|\Acc(w_0)|$, is realizable by some $\mathsf{K}$-frame.
\end{proposition}

The centered symmetry group of a $\mathsf{K}$-frame may range from $\Sym(\Acc(w_0))$ (maximal symmetry, classical exchangeability) to the trivial group (no exchangeability at all). The representation is \emph{frame-relative}: to know the form of admissible credence, one must know the specific frame, not merely that it validates $\mathsf{K}$.

Without at least $\mathsf{S4}$-style structure, modal axioms alone do not constrain credal pooling.

\subsection{Summary: the modal hierarchy}

The following table summarizes the relationship between modal logic, frame condition, orbit structure, and probabilistic consequence.

\begin{center}
\small
\footnotesize
\begin{tabular}{@{}lllllp{3.0cm}@{}}
\toprule
\textbf{Logic} & \textbf{Frame} & \textbf{Orbits} & \textbf{Symmetry} & \textbf{Learning} & \textbf{Probabilistic form} \\
\midrule
$\mathsf{S5}$\textsuperscript{$\dagger$} & Equivalence & Single & Maximal & Pooled & Global de Finetti \\
$\mathsf{S4}$ & Preorder & Multiple (poss.) & Frame-dep. & Orbit-local & Orbit-local de Finetti on $\Orbit^\infty$; weaker elsewhere \\
$\mathsf{T}$ & Reflexive & Frame-rel. & Minimal & Frame-dep. & No canonical form \\
$\mathsf{K}$ & Arbitrary & Frame-rel. & Arbitrary & Frame-dep. & No canonical form \\
\bottomrule
\multicolumn{6}{@{}l}{\textsuperscript{$\dagger$}Point-homogeneous, with (Ext).}
\end{tabular}
\end{center}

\section{Philosophical Interpretation}
\label{sec:philosophy}

We develop two main applications: what the representation implies for rational indifference and learning, and how the orbit structure may bear on fine-grained credal differentiation.

\subsection{Symmetry and rational indifference}

Modal exchangeability is a principle of rational indifference adapted to structured logical space. Classical indifference principles assign equal probability to ``structurally similar'' outcomes. In a flat setting, structural similarity means permutation equivalence, and indifference yields exchangeability. In the modal setting, structural similarity means equivalence under accessibility-preserving automorphisms that fix the base world. The representation theorems show that modal indifference yields within-orbit identical conditional marginals in general, and orbit-wise mixture-of-i.i.d.\ structure only on those orbits satisfying (Ext) and countable infinitude.

This is not an unrestricted principle of indifference of the kind that faces classical objections (Bertrand's paradox, language-dependence). It is a principle adapted to a \emph{given} structure: the frame determines the symmetry group, and the symmetry group determines the scope of indifference. There is no free choice of partition or measure of similarity; the accessibility relation and the base world fix both.

The strength of the representation depends on the strength of the indifference principle, which in turn depends on the size of the stabilizer. Under point-homogeneous $\mathsf{S5}$, the stabilizer is maximal, and (with (Ext) and countable infinitude) indifference yields a single global parameter. Under $\mathsf{S4}$, the stabilizer is smaller; the representation may fragment, and orbits outside $\Orbit^\infty$ yield only weaker invariant structure. Under $\mathsf{K}$ or $\mathsf{T}$, the stabilizer may be trivial and indifference imposes almost nothing. The modal hierarchy is thus also a hierarchy of rational indifference.

\subsection{Local versus global learning}

For orbits in $\Orbit^\infty$, the orbit structure governs what an agent can learn from observation. Observations at worlds within such an orbit $O$ update the directing measure $\Lambda_O$ but are uninformative about $\Lambda_{O'}$ for $O' \in \Orbit^\infty$ with $O' \neq O$ unless the agent's prior couples them. Rigidity ensures that learning from any world $w \in O$ is equivalent to learning from any other $w' \in O$. Cross-orbit transfer requires additional structure---either in the agent's prior or in the group action.

Under point-homogeneous $\mathsf{S5}$ with (Ext) and countable infinitude (single orbit in $\Orbit^\infty$), all observations pool: every piece of evidence bears on the same global parameter, and the agent's posterior concentrates as data accumulates (under standard regularity conditions for Bayesian consistency). Under $\mathsf{S4}$ with multiple orbits in $\Orbit^\infty$, learning is orbit-local by default. Hierarchical priors coupling orbit parameters along the accessibility relation can encode directed information transfer, but this is a modeling choice orthogonal to the representation theorem.

The conceptual separation: \emph{what symmetry yields}---within-orbit rigidity, identical conditional marginals, and (under (Ext) and countable infinitude) conditional i.i.d.\ given the directing measure; \emph{what prior modeling choices add}---cross-orbit coupling, hierarchical parameter structure, and the scope of information transfer.

A note on dynamics is in order. Conditioning on observations at specific worlds in general \emph{breaks} invariance under automorphisms that move those worlds, so the posterior is not automatically modally exchangeable at $w_0$. Modal exchangeability is preserved only when conditioning on $G$-invariant evidence, or---more flexibly---only relative to the stabilizer subgroup $G' = \Stab_{\{w_0\} \cup W_{\mathrm{obs}}}$ that fixes both $w_0$ and the set of observed worlds. Under this relative symmetry, the posterior admits an orbit-wise representation with respect to $G'$, where the $G'$-orbits are typically finer than the original $G$-orbits. A fully developed dynamic theory, connecting to relative exchangeability and to exchangeability-preserving update rules in dynamic epistemic logic, lies beyond the scope of this paper.

Two formal consequences make the learning structure precise:

\begin{corollary}[Orbit-local updating]
\label{cor:orbit-local}
Let $P$ be modally exchangeable and let $E \in \mathcal{F}$ be an event that is \emph{orbit-measurable} for orbit $O$---that is, $E$ depends only on valuations at worlds in $O$. If the prior factorizes across orbits conditional on $\Xi$ \textup{(}i.e., $P(\cdot \mid \Xi) = \bigotimes_{O' \in \Orbit} P_{O'}(\cdot \mid \Xi)$\textup{)}, then the posterior for every other orbit is unchanged:
\[
P(\{V(w)\}_{w \in O'} \in \cdot \mid E, \, \Xi) = P(\{V(w)\}_{w \in O'} \in \cdot \mid \Xi) \quad \text{for all } O' \neq O.
\]
\end{corollary}

\begin{proof}
Under the factorization assumption, the conditional distribution of $\{V(w)\}_{w \in O'}$ given $\Xi$ is independent of $\{V(w)\}_{w \in O}$. Since $E$ is $O$-measurable, conditioning on $E$ does not update the $O'$-marginal.
\end{proof}

\begin{corollary}[Symmetry-induced pooling]
\label{cor:pooling}
Symmetry alone yields universal pooling---a single directing measure governing all non-base accessible worlds---if and only if the centered symmetry group $G$ acts transitively on $\Acc(w_0)$, satisfies \textup{(Ext)}, and $\Acc(w_0)$ is countably infinite. In that case there is a single orbit lying in $\Orbit^\infty$, and the representation reduces to the point-homogeneous $\mathsf{S5}$ case \textup{(Theorem~\ref{thm:S5})}.

When multiple orbits exist, modal exchangeability is consistent with factorized priors that yield no cross-orbit pooling \textup{(Corollary~\ref{cor:orbit-local})}. Cross-orbit pooling can still occur through hierarchical prior coupling of the orbit-specific directing measures, but such coupling is an additional modeling choice, not a consequence of symmetry.
\end{corollary}

\begin{proof}
If $G$ acts transitively with (Ext), $\Acc(w_0)$ is a single orbit and Theorem~\ref{thm:S5} applies: all evidence updates the unique directing measure. Conversely, if there are multiple orbits, the factorized prior of Corollary~\ref{cor:orbit-local} is consistent with modal exchangeability and yields no cross-orbit pooling; any pooling that does occur is mediated by the joint prior over $\{\Lambda_O\}$, which symmetry does not determine.
\end{proof}

\subsection{Rigidity and probabilistic fine-graining}
\label{subsec:rigidity}

Rigidity is the central structural consequence of modal exchangeability on orbits in $\Orbit^\infty$. It asserts that for each such orbit, the directing measure is constant across that orbit: $\Lambda_O$ takes the same value at every world in $O$. The derivation is short. The $G$-invariance of $P$ is the sole premise. The directing measure $\Lambda_O$ is determined by the distribution of $\{V(w)\}_{w \in O}$. For any $\pi \in G$ with $\pi(O) = O$, $G$-invariance implies $\text{Law}(\{V(\pi(w))\}_{w \in O}) = \text{Law}(\{V(w)\}_{w \in O})$, so $\Lambda_O$ is the same whether computed from $\{V(w)\}$ or from $\{V(\pi(w))\}$. Formally, $\Lambda_O$ is $\FG$-measurable: $\Lambda_O(\pi \cdot V) = \Lambda_O(V)$ $P$-a.s.\ for all $\pi \in G$.

What rigidity rules out is world-relative parameterization within a single orbit in $\Orbit^\infty$. Without the invariance requirement, one could assign a different directing measure $\Lambda_w$ for each $w \in O$. Rigidity eliminates this: structurally indistinguishable worlds must share a common governing parameter. Across orbits in $\Orbit^\infty$, parameters \emph{can} differ; for orbits outside $\Orbit^\infty$, no directing measure is delivered at all.

A critical clarification: rigidity is orbit-relative, not global. It means that on orbits in $\Orbit^\infty$, the directing measure is invariant under the stabilizer's action within its orbit---not that it is constant across all worlds in the frame or across distinct orbits.

\subsection{Orbit structure and structural credal fine-graining}
\label{subsec:hyperintensionality}

The framework supports \emph{structural fine-graining of credal attitudes indexed by orbit regions}. This yields differentiation across structurally distinct regions of the frame, rather than hyperintensionality in the strict sense of distinguishing coextensive propositions. A brief orientation to the broader hyperintensionality literature appears in \citet{Berto2022}; the present framework engages that literature obliquely, through a structural mechanism rather than a semantic one.

The mechanism is the following. Distinct orbits in $\Orbit^\infty$ carry separately parameterized rigid directing measures $\Lambda_{O_1}, \Lambda_{O_2}$. Propositions or inquiry-targets indexed to different orbits---rather than evaluated purely extensionally over $\Acc(w_0)$ as a whole---can therefore receive divergent credal treatment. An agent evaluating a proposition about worlds in $O_1$ and a structurally parallel proposition about worlds in $O_2$ can rationally hold different credences, since $G$-invariance imposes no constraint relating $\Lambda_{O_1}$ and $\Lambda_{O_2}$. For orbits outside $\Orbit^\infty$, only weaker invariant structure (identical conditional marginals, possibly stationarity) supports any fine-graining. The indexing to orbits is what makes the differentiation possible; a single proposition that picks out the same set of worlds across $\Acc(w_0)$ still receives a single credence.

This is structural credal fine-graining rather than hyperintensionality in the strict sense: the orbit decomposition does not by itself distinguish propositions that are literally coextensive at the level of world-sets. It is a complementary mechanism, not a replacement for approaches that enrich the semantic objects.

\begin{proposition}[Orbit-sensitive credal differentiation]
\label{prop:hyperintensional}
Let $O_1, O_2 \in \Orbit^\infty$ be distinct orbits in $\Acc(w_0)$. Structurally parallel claims or inquiry-targets indexed to $O_1$ and $O_2$ can rationally receive different credal treatment, captured by separate rigid parameters $\Lambda_{O_1}$ and $\Lambda_{O_2}$ whose joint distribution is not determined by $G$-invariance.
\end{proposition}

\begin{proof}
By Theorem~\ref{thm:S4}, no element of $G$ maps $O_1$ to $O_2$, so $G$-invariance imposes no constraint relating $\Lambda_{O_1}$ to $\Lambda_{O_2}$. The joint prior over $(\Lambda_{O_1}, \Lambda_{O_2})$ may assign arbitrary marginals. An agent can therefore hold sharply differentiated credal profiles across the two orbit-regions while satisfying modal exchangeability throughout.
\end{proof}

The mechanism is offered as one route toward credal fine-graining generated by the interaction of modal structure and probabilistic symmetry on a Kripke frame.

\section{Examples}
\label{sec:examples}

\subsection{S4 fragmentation: epistemic regions with local learning}

\begin{example}[Two-orbit S4 frame with non-isomorphic regions]
\label{ex:S4}
Let $W = \{w_0\} \cup \{a_i : i \in \N\} \cup \{b_i : i \in \Z\}$. Define accessibility as follows. The base world $w_0$ is reflexive and accesses every other world. The $a$-region is a complete reflexive graph: $R(a_i, a_j)$ for all $i,j \in \N$. The $b$-region is a reflexive linear order indexed by $\Z$: $R(b_i, b_j)$ iff $j \geq i$. No cross edges hold between regions. The frame is reflexive and transitive, so it validates $\mathsf{S4}$.

The non-base accessible set is $\Acc(w_0) = \{a_i\} \cup \{b_i\}$. The two regions are \emph{not} isomorphic as relational structures: in the $a$-region the relation is symmetric on distinct worlds, whereas in the $b$-region it is a strict linear order. No automorphism of $\langle W, R\rangle$ fixing $w_0$ can map an $a$-world to a $b$-world, since the local relational profile differs. The stabilizer decomposes as
\[
G \cong \Sym(\{a_i\}_{i \in \N}) \times \Aut(\Z, \leq),
\]
where $\Aut(\Z, \leq)$ is the group of order-preserving bijections of $\Z$, i.e.\ the translations $\pi_k(b_i) = b_{i+k}$. This yields two orbits: $O_a = \{a_i\}$ and $O_b = \{b_i\}$. The $\Z$-indexing is essential: on $\N$ the order-preserving bijections reduce to the identity and the $b_i$ would fail to form a single orbit.

The $a$-orbit satisfies (Ext): $G|_{O_a} = \Sym(O_a)$. The $b$-orbit does not: translations do not realize arbitrary finite permutations (swapping $b_0$ and $b_1$ would reverse order). By Theorem~\ref{thm:general}, any modally exchangeable $P$ therefore exhibits two qualitatively different orbit behaviours. Within $O_a$, the valuations are conditionally i.i.d.\ with a rigid directing measure $\Lambda_{O_a}$. Within $O_b$, the valuations are identically distributed conditional on $\Xi$, but the conditional law is shift-stationary rather than i.i.d.; there is no single directing measure $\Lambda_{O_b}$ in the de Finetti sense.

The epistemic interpretation: the agent faces two structurally distinct regions. In the $a$-region, observations inform a single rigid parameter in the classical de Finetti manner. In the $b$-region, observations inform the shift-stationary conditional law but do not pool into a single directing parameter. This illustrates the two mechanisms by which modal structure shapes the representation: (i) the orbit decomposition distinguishes $O_a$ from $O_b$, and (ii) the internal structure of each orbit controls whether the representation is i.i.d.\ (under (Ext) and countable infinitude) or merely stationary.
\end{example}

\subsection{(Ext) failure revisited: the $b$-orbit as a stationary non-i.i.d.\ representation}

The $b$-region of Example~\ref{ex:S4} already exhibits (Ext) failure, and it is worth recording what the invariant latent structure delivers there. By Theorem~\ref{thm:general}(2), the valuations $\{V(b_i)\}_{i \in \Z}$ are conditionally identically distributed given $\Xi$. The conditional distribution is invariant under translations but not under arbitrary finite permutations, so Hewitt--Savage does not apply: the representation yields a shift-stationary process on $\Z$ rather than a mixture of i.i.d.\ draws. Stationarity without independence is a familiar object in probability theory, and its appearance here marks the precise boundary between the two representation regimes identified in Proposition~\ref{prop:ext-role}: conditional identical marginals (always, under $G$-invariance) versus conditional i.i.d.\ (requiring (Ext) and countable infinitude). Property (Ext) is thus not a technical decoration but the mathematical content of the transition.

\section{Remaining Proofs and Technical Notes}
\label{sec:proofs}

The proof of the general theorem (Theorem~\ref{thm:general}) was given in Section~\ref{sec:general}. Here we collect the key lemmas in explicit form and note the constructive content of the proofs.

\subsection{Sufficient conditions for the extension property}

\begin{proposition}[Sufficient conditions for (Ext)]
\label{prop:ext-sufficient}
Property (Ext) holds on an orbit $O$ in the following cases:

\begin{enumerate}[leftmargin=2em]
\item If $G|_O = \Sym(O)$ (the restricted action is the full symmetric group), then (Ext) holds trivially.

\item More generally, (Ext) holds whenever the relational structure induced on $O$ by $R$ is ultrahomogeneous in the Fra\"{i}ss\'{e} sense---that is, any partial isomorphism of finite subsets of $O$ extends to an automorphism in $G|_O$.

\item In particular, if the induced subframe on $O$ carries only reflexive accessibility and all worlds in $O$ have identical incidence patterns to worlds outside $O$, then $G|_O = \Sym(O)$ and (Ext) holds.
\end{enumerate}
\end{proposition}

\begin{proof}
Case 1 is immediate. Case 2 follows from Fra\"{i}ss\'{e}'s characterization: ultrahomogeneity means every partial isomorphism extends to a full automorphism. Case 3: if internal structure is trivial and external incidence is uniform across $O$, any permutation of $O$ preserves all accessibility relations, so $G|_O = \Sym(O)$.
\end{proof}

\subsection{Key lemmas}

\begin{lemma}[Ergodic decomposition]
\label{lem:disintegration}
Let $P$ be modally exchangeable. There exists a $\FG$-measurable random variable $\Xi$ and a regular conditional probability $P(\cdot \mid \Xi)$ such that $P = \int P(\cdot \mid \Xi = \xi) \, P_\Xi(d\xi)$, with each fiber $G$-ergodic.
\end{lemma}

\begin{proof}
Disintegration over $\FG$ on the standard Borel space $(\Omega, \mathcal{F})$ exists by~\citep[Theorem 452I]{Fremlin2006}, requiring only dependent choice (ZF + DC).
\end{proof}

\begin{lemma}[Identical distribution within orbits]
\label{lem:identical}
For any $v, v'$ in the same $G$-orbit, $\Law(V(v) \mid \Xi) = \Law(V(v') \mid \Xi)$ holds $P$-a.s.
\end{lemma}

\begin{proof}
There exists $\pi \in G$ with $\pi(v) = v'$. By $G$-invariance of $P$ and $\pi$-invariance of $\Xi$:
$P(V(v) \in B \mid \Xi) = P((\pi \cdot V)(v') \in B \mid \Xi) = P(V(v') \in B \mid \Xi)$.
\end{proof}

\begin{lemma}[Exchangeability within orbits]
\label{lem:exch}
Fix an orbit $O$ and assume $G|_O$ has property \textup{(Ext)}. Then for $P$-a.e.\ $\Xi$, the conditional distribution of $\{V(w)\}_{w \in O}$ is exchangeable.
\end{lemma}

\begin{proof}
Property (Ext) ensures that for any finite permutation $\sigma$ of $O$, there exists $\pi \in G$ with $\pi|_O = \sigma$. By $G$-ergodicity of $P(\cdot \mid \Xi)$, the conditional distribution is invariant under all such permutations.
\end{proof}

\subsection{Constructive content and precision on disintegration}

All proofs use only ZF + DC (Zermelo--Fraenkel set theory with dependent choice). The critical step is disintegration on standard Borel spaces, which requires dependent choice. Since $W$ is countable and $L$ is finite, $(\Omega, \mathcal{F})$ is a standard Borel space without any appeal to choice. The disintegration theorem applies, yielding a regular conditional probability $P(\cdot \mid \FG)$ by~\citep[Theorem 452I]{Fremlin2006}.

Computability of the orbit structure depends on effective presentations of the frame and access to its automorphism group; we do not pursue general computability claims here. Under suitable effective presentations, the directing measures can be shown to be computable in the sense of \citet{FreerRoy2009}.

\section{Conclusion}
\label{sec:conclusion}

Modal exchangeability---$G$-invariance of a probability measure on a Kripke frame, centered at a designated world---shapes credal architecture via the orbit geometry of the centered symmetry group. Under point-homogeneous $\mathsf{S5}$ with (Ext) and countable infinitude, one recovers a single de Finetti parameter and globally pooled learning. Under $\mathsf{S4}$, the representation may fragment into multiple orbits; those in $\Orbit^\infty$ admit local de Finetti theorems with rigid directing measures, while other orbits receive only weaker invariant structure.

The philosophical suggestion: modal accessibility structure constrains how rational uncertainty is organized via the orbit decomposition---how many independent parameters credence carries on the orbits that support them, which worlds share a parameter, and where learning pools. The orbit decomposition also supports structural credal fine-graining indexed to orbit regions, a mechanism distinct from hyperintensionality in the strict sense of distinguishing coextensive propositions.

Several extensions remain open: exchangeability-preserving update rules in dynamic modal logic~\citep{vanBenthem2011}, approximate or partial exchangeability under subgroups of the stabilizer, and the question of which frame classes beyond the standard hierarchy support a canonical class-level representation. The assumption of countable $W$ and finite $L$ ensures standard Borel structure; uncountable extensions would require stronger choice principles. The framework is offered as a contribution to the ongoing conversation between probabilistic symmetry theory and modal epistemology---a conversation in which the structure of possibility and the structure of rational uncertainty turn out to be more tightly coupled than either tradition has recognised on its own.


\appendix

\section{Comparison Table: Probabilistic Symmetry Frameworks}
\label{app:representation}

The following table provides a concise comparison of the major probabilistic symmetry frameworks and their relation to the present paper.

\begin{center}
\footnotesize
\begin{tabular}{@{}p{2.0cm}p{1.9cm}p{2.3cm}p{2.3cm}p{3.3cm}@{}}
\toprule
\textbf{Framework} & \textbf{Random object} & \textbf{Symmetry group} & \textbf{Repres.\ form} & \textbf{Relation to this paper} \\
\midrule
de Finetti / Hewitt--Savage & Sequence $(X_n)$ & $\Sym(\N)$ & Mixture of i.i.d. & Recovered as homogeneous $\mathsf{S5}$ case \\[4pt]
Aldous--Hoover & Array $(X_{ij})$ & $\Sym(\N)^2$ & Function of i.i.d.\ uniforms & Broader background; different object and group \\[4pt]
Austin--Panchenko & Hierarchically indexed variables & Tree automorphisms & Hierarchical conditionally i.i.d. & Closer in spirit; hierarchy endogenous here \\[4pt]
Crane--Towsner & Random relational structure & $\Aut(\mathcal{M})$ & Definable-closure decomposition & Nearest abstract analogue; modal specialization here \\[4pt]
\textbf{Present paper} & Valuation on Kripke frame & $\Stab_{w_0} \leq \Aut\langle W,R\rangle$ & Orbit-wise mixture, rigid $\Lambda_O$ & Modal interpretation, philosophical content \\
\bottomrule
\end{tabular}
\end{center}

\bibliographystyle{plainnat}
\bibliography{references}

@article{CraneTowsner2018,
  author    = {H. Crane and H. Towsner},
  title     = {Relatively exchangeable structures},
  journal   = {Journal of Symbolic Logic},
  volume    = {83},
  number    = {2},
  pages     = {416--442},
  year      = {2018}
}

@article{deFinetti1937,
  author    = {B. de Finetti},
  title     = {La pr\'{e}vision: ses lois logiques, ses sources subjectives},
  journal   = {Annales de l'Institut Henri Poincar\'{e}},
  volume    = {7},
  pages     = {1--68},
  year      = {1937}
}

@inproceedings{FreerRoy2009,
  author    = {C. E. Freer and D. M. Roy},
  title     = {Computable exchangeable sequences have computable de {F}inetti measures},
  booktitle = {CiE 2009},
  series    = {LNCS},
  volume    = {5635},
  pages     = {218--231},
  publisher = {Springer},
  year      = {2009}
}

@book{Fremlin2006,
  author    = {D. H. Fremlin},
  title     = {Measure Theory, Volume 4: Topological Measure Spaces},
  publisher = {Torres Fremlin},
  year      = {2006}
}

@article{Hewitt1955,
  author    = {E. Hewitt and L. J. Savage},
  title     = {Symmetric measures on {C}artesian products},
  journal   = {Transactions of the American Mathematical Society},
  volume    = {80},
  pages     = {470--501},
  year      = {1955}
}

@incollection{Berto2022,
  author    = {F. Berto and D. Nolan},
  title     = {Hyperintensionality},
  booktitle = {The {Stanford} Encyclopedia of Philosophy (Summer 2021 Edition)},
  editor    = {E. N. Zalta},
  publisher = {Metaphysics Research Lab, Stanford University},
  year      = {2021}
}

@article{Fraisse1954,
  author    = {R. Fra\"{\i}ss\'{e}},
  title     = {Sur l'extension aux relations de quelques propri\'{e}t\'{e}s des ordres},
  journal   = {Annales Scientifiques de l'\'{E}cole Normale Sup\'{e}rieure},
  volume    = {71},
  pages     = {363--388},
  year      = {1954}
}

@book{Hodges1993,
  author    = {W. Hodges},
  title     = {Model Theory},
  publisher = {Cambridge University Press},
  year      = {1993}
}

@book{TentZiegler2012,
  author    = {K. Tent and M. Ziegler},
  title     = {A Course in Model Theory},
  publisher = {Cambridge University Press},
  year      = {2012}
}

@book{vanBenthem2010,
  author    = {J. van Benthem},
  title     = {Modal Logic for Open Minds},
  publisher = {CSLI Publications},
  year      = {2010}
}

@book{vanBenthem2011,
  author    = {J. van Benthem},
  title     = {Logical Dynamics of Information and Interaction},
  publisher = {Cambridge University Press},
  year      = {2011}
}

@article{Diaconis1977,
  author    = {P. Diaconis},
  title     = {Finite forms of de {F}inetti's theorem on exchangeability},
  journal   = {Synthese},
  volume    = {36},
  number    = {2},
  pages     = {271--281},
  year      = {1977}
}

@article{SuppesZanotti1981,
  author    = {P. Suppes and M. Zanotti},
  title     = {When are probabilistic explanations possible?},
  journal   = {Synthese},
  volume    = {48},
  number    = {2},
  pages     = {191--199},
  year      = {1981}
}

@article{Bueno2009,
  author    = {O. Bueno and S. A. Shalkowski},
  title     = {Modalism and Logical Pluralism},
  journal   = {Mind},
  volume    = {118},
  number    = {470},
  pages     = {295--321},
  year      = {2009}
}

\end{document}